\documentclass[12pt]{amsart}

\usepackage{amsmath,amsfonts,amsthm,color,vector,epsfig}
\usepackage{pstricks}
\usepackage{tikz}

\renewcommand{\theequation}{\thesection.\arabic{equation}}
\newtheorem{theorem}{Theorem}[section]
\newtheorem{theorem-lettre}{Theorem}
\newtheorem{theorem-alpha}{Theorem}
\newtheorem{lemma}[theorem]{Lemma}
\newtheorem{proposition}[theorem]{Proposition}

\newcommand{\eqnsection}{
\renewcommand{\theequation}{\thesection.\arabic{equation}}
    \makeatletter
    \csname  @addtoreset\endcsname{equation}{section}
    \makeatother}
\eqnsection
\def\e{{\mathbb E}}
\def\p{{\mathbb P}}
\def\z{{\mathbb Z}}
\def\n{{\mathbb N}}

\def\ee{\mathrm{e}}
\def\R{{\mathbb R}}
\def\Z{{\mathbb Z}}

\def\N{{\mathbb N}}
\def\E{{\mathbb E}}

\def\F{\mathcal F}

\author[L.-P. Arguin]{Louis-Pierre ARGUIN}
\thanks{L.-P. A. is supported by a NSERC discovery grant and a grant FQRNT {\it Nouveaux chercheurs}.}
\address{D\'epartement de Math\'ematiques et Statistique, Universit\'e de Montr\'eal, Montr\'eal, H3T 1J4, Canada}
\email{arguinlp@dms.umontreal.ca}

\author[O. Zindy]{Olivier ZINDY}
\thanks{O.Z. is partially supported by the french ANR project MEMEMO2 2010 BLAN 0125.}
\address{Laboratoire de Probabilit\'es et Mod\`eles Al\'eatoires, CNRS UMR 7599, Universit\'e Paris 6, 4
place Jussieu, 75252 Paris Cedex 05, France}
\email{olivier.zindy@upmc.fr}


\title[Poisson-Dirichlet statistics and 2D Gaussian Free Field]{Poisson-Dirichlet Statistics for the extremes of the two-dimensional discrete Gaussian Free Field}

\date{6 October, 2013}

\begin{document}

\maketitle

\bigskip

{\footnotesize \noindent{\slshape\bfseries Abstract.} 
In a previous paper, the authors introduced an approach to prove that the statistics of the extremes of a log-correlated Gaussian field converge 
to a Poisson-Dirichlet variable at the level of the Gibbs measure at low temperature and under suitable test functions.
The method is based on showing that the model admits a one-step replica symmetry breaking in spin glass terminology.
This implies Poisson-Dirichlet statistics by general spin glass arguments.
%and that this must imply in turns Poisson-Dirichlet statistics. 
In this note, this approach is used to prove Poisson-Dirichlet statistics for the two-dimensional discrete Gaussian free field, where boundary effects demand a more delicate analysis.
}

\bigskip
\bigskip

%%%%%%%%%%%%%%%%%%%%%%%%%%%%%%%%%%%%%%%%%%%%%%%%%%%%%%%%%%%%%%%%%%%%%%%%%%%%%

\section{Introduction}

\subsection{The model}

Consider a finite box $A$ of $\Z^2$. 
The Gaussian free field (GFF) on $A$ with Dirichlet boundary condition is the centered Gaussian field $(\phi_v, v\in A)$ with the covariance matrix
\begin{equation}
\label{eqn: cov}
G_A(v,v'):= E_v\left[\sum_{k=0}^{\tau_{A}} 1_{v'}(S_k)\right]\ ,
\end{equation}
where $(S_k, k\geq 0)$ is a simple random walk with $S_0=v$ of law $P_v$
killed at the first exit time of $A$, $\tau_{A}$, i.e.~ the first time where the walk reaches the boundary $\partial A$. 
Throughout the paper, for any $A\subset \Z^2$, $\partial A$ will denote the set of vertices in $A^c$ that share an edge with a vertex of $A$.
We will write $\p$ for the law of the Gaussian field and $\E$ for the expectation.
For $B\subset A$, we denote the $\sigma$-algebra generated by $\{\phi_v, v\in B\}$ by $\F_B$.

We are interested in the case where $A=V_N:=\{1,\dots, N\}^2$ in the limit $N\to\infty$. 
For $0\leq \delta < 1/2$, we denote by $V_N^\delta$ the set of the points of $V_N$ whose distance to the boundary $\partial V_N$
is greater than $\delta N$. 
In this set, the variance of the field diverges logarithmically with $N$, cf. Lemma \ref{lem: green estimate} in the appendix, 
\begin{equation}
\label{eqn: variance}
\E[\phi^2_v]=G_{V_N}(v,v)= \frac{1}{\pi} \log N^2 + O_N(1), \qquad { \forall v \in V_N^\delta},
\end{equation}
where $O_N(1)$ will always be a term which is uniformly bounded in $N$ and in $v\in V_N$. 
(The term $o_N(1)$ will denote throughout a term which goes to $0$ as $N\to\infty$ uniformly in all other parameters.)
Equation \eqref{eqn: variance} follows from the fact that for $v\in V_N^\delta$ and $u\in \partial V_N$, $\delta N \leq \| v-u \| \leq \sqrt{2}(1-\delta) N$, where $\|\cdot\|$ denotes the Euclidean norm on $\Z^2$.
A similar estimate yields an estimate on the covariance
\begin{equation}
\label{eqn: covariance}
 \E[\phi_v\phi_{v'}]=G_{V_N}(v,v')=\frac{1}{\pi} \log \frac{N^2}{\|v-v'\|^{ 2}} +O_N(1), \qquad  \forall v,v'\in V_N^\delta .
\end{equation}
In view of \eqref{eqn: variance} and \eqref{eqn: covariance}, the Gaussian field $(\phi_v,v\in V_N)$ is said to be {\it log-correlated}.
On the other hand, there are many points that are outside $V_N^{\delta}$ (of the order of $N^2$ points) for which the estimates \eqref{eqn: variance} and \eqref{eqn: covariance} are not correct. 
Essentially, the closer the points are to the boundary the lesser are the variance and covariance as the simple random walk in \eqref{eqn: cov} has a higher probability of exiting $V_N$ early. 
This decoupling effect close to the boundary complicates the analysis of the extrema of the GFF by comparison with log-correlated Gaussian fields with stationary distribution.

%%%%%%%%%%%%%%%%%%%%%%%%%%%%%%%%%%%%%%%%%%%%%%%%%%%%%%%%%%%%%%%%%%%%%

\bigskip

\subsection{Main results}
It was shown by Bolthausen, Deuschel, and Giacomin \cite{bolthausen-deuschel-giacomin} that the maximum of the GFF in $V_N^\delta$ satisfies

\begin{equation}
\lim_{N\to\infty} \frac{\max_{v\in V_N^\delta} \phi_v }{\log N^2}= \sqrt{\frac{2}{\pi}}, \qquad  \text{ in probability.}
\end{equation}
A comparison argument using Slepian's lemma can be used to extend the result to the whole box $V_N$.
Their technique was later refined by Daviaud \cite{daviaud} who computed the {\it log-number of high points} in $V_N^\delta$: for $0<\lambda<1$,

\begin{equation}
\label{eqn: daviaud}
\lim_{N\to\infty} \frac{1}{\log N^2}  \log \#\{v\in V_N^\delta: \phi_v \,  \geq \, \lambda  \sqrt{\frac{2}{\pi}} \log N^2\}= 1-\lambda^2, \qquad \text{ in probability.}
\end{equation}

It is a simple exercise to show using the above results that the {\it free energy} in $V_N$ of the model is given by
\begin{equation}
\label{eqn: free energy GFF}
f(\beta):=\lim_{N\to\infty}\frac{1}{\log N^2}\log \sum_{v\in V_N}e^{\beta \phi_v}=
\begin{cases}
1+\frac{\beta^2 }{2\pi}, &\text{ if $\beta\leq \sqrt{2\pi}$,}\\
\sqrt{\frac{2}{\pi}}\beta ,  &\text{ if $\beta\geq  \sqrt{2\pi}$,}
\end{cases}
\qquad
 \text{ a.s. and in $L^1$. }
\end{equation}
Indeed, there is the clear lower bound $\log \sum_{v\in V_N}e^{\beta \phi_v} \geq \log \sum_{v\in V_N^\delta}e^{\beta \phi_v}$,
which can be evaluated using the log-number of high points \eqref{eqn: daviaud} by Laplace's method.
The upper bound is obtained using a comparison argument with i.i.d. centered Gaussians.

A striking fact is that the three above results correspond to the expressions for $N^2$ independent Gaussian variables of variance $\frac{1}{\pi}\log N^2$. 
In other words, correlations have no effects on the above observables of the extremes.
The purpose of the paper is to extend this correspondence to observables related to the Gibbs measure.

To this aim, consider the {\it normalized Gibbs weights} or {\it Gibbs measure}
\begin{equation*}
{\mathcal G}_{\beta, N}(\{v\}):=\frac{\ee^{\beta \phi_v}}{Z_N(\beta)}, \qquad v \in  V_N ,
\end{equation*}
where $Z_N(\beta):= \sum_{v\in V_N} \ee^{\beta \phi_v}$.
We consider the normalized covariance or {\it overlap}
\begin{equation}
\label{eqn: q}
q(v,v'): =\frac{\E[\phi_v\phi_{v'}]}{\frac{1}{\pi}\log N^2}, \qquad \forall v,v' \in  V_N.
\end{equation}
This is the covariance divided by the dominant term of the variance in the bulk. 

In spin glasses, the relevant object to classify the extreme value statistics of strongly correlated variables is the {\it two-overlap distribution function}
\begin{equation}
\label{eqn: x}
x_{\beta,N}(q):=  \e \left[  {\mathcal G}_{\beta, N}^{\times 2} \left\{q(v,v') \le q \right\} \right], \qquad  0\leq q\leq 1 .
\end{equation}
The main result shows that the 2D GFF falls within the class of models that exhibit a {\it one-step replica symmetry breaking} at low temperature.
\begin{theorem}
\label{thm: overlap}
For $\beta > \beta_c=\sqrt{2\pi}$,
 $$
\lim_{N\to\infty }x_{\beta,N}(r):= \lim_{N\to\infty } \e \left[  {\mathcal G}_{\beta, N}^{\times 2} \left\{q(v,v') \le q \right\} \right]=
\begin{cases}
\frac{\beta_c}{\beta} &\text{ for $0 \le r <1$,}\\
1 &\text{ for $r=1$.}
\end{cases}
$$
\end{theorem}
Note that for $\beta\leq \beta_c$, it follows from \eqref{eqn: free energy GFF} that the overlap is $0$ almost surely.
The result is the analogue for the 2D GFF of the results obtained by Derrida \& Spohn \cite{derrida-spohn} and Bovier \& Kurkova \cite{bovier-kurkova1,bovier-kurkova2} 
for the branching Brownian motion and for GREM-type models.
In \cite{arguin-zindy}, such a result was proved for a non-hierarchical log-correlated Gaussian field constructed
from the multifractal random measure of Bacry \& Muzy \cite{bacry-muzy}, see also \cite{bouchaud-fyodorov} for a closely related model.
This type of result was conjectured by Carpentier \& Ledoussal \cite{carpentier-ledoussal}. 
We also remark that Theorem \ref{thm: overlap} shows that at low temperature two points sampled with the Gibbs measure have overlaps $0$ or $1$. 
This is consistent with the result of Ding \& Zeitouni \cite{ding-zeitouni} who showed that the extremal values of GFF are at distance from each other of order one or of order $N$. 

A general method to prove Poisson-Dirichlet statistics for the distribution of the overlaps from the one-step replica symmetry breaking was laid down in \cite{arguin-zindy}.
This connection is done via the (now fundamental) Ghirlanda-Guerra identities. 
Another equivalent approach would be using {\it stochastic stability} as developed in \cite{aizenman-contucci, arguin, arguin-chatterjee}.
The reader is referred to Section 2.3 of \cite{arguin-zindy} where the connection is explained in details for general Gaussian fields.
For the sake of conciseness, we simply state the consequence for the 2D GFF.

Consider the product measure ${\mathcal G}_{\beta,N}^{\times s}$ on $s$ {\it replicas} $(v_1,\dots,v_s)\in V_N^{\times s}$.
Let $F:[0,1]^{\frac{s(s-1)}{2}}\to \R$ be a continuous function.
Write  $F(q_{ll'})$ for the function evaluated at $q_{ll'}:=q(v_l,v_{l'})$, $l\neq l'$, for $(v_1,\dots ,v_s)\in  V_N^{\times s}$.
We write $\E {\mathcal G}_{\beta,N}^{\times s}\big(F(q_{ll'})\big)$ for the averaged expectation.
Recall that a {\it Poisson-Dirichlet variable} $\xi$ of parameter $\alpha$ is a random variable on the space of decreasing weights $\vec{s}=(s_1,s_2,\dots)$ with $1\geq s_1\geq s_2\geq \dots\geq 0$ and $\sum_{i}s_i\leq 1$ which has the same law as $\left(\eta_i/\sum_j\eta_j, i\in \N\right)_\downarrow$ where  $\downarrow$ stands for the decreasing rearrangement and $\eta=(\eta_i,i\in\N)$ are the atoms of a Poisson random measure on $(0,\infty)$ of intensity measure $s^{-\alpha-1} ~ ds$.

The theorem below is a direct consequence of the Theorem \ref{thm: overlap}, the differentiability of the free energy \eqref{eqn: free energy GFF} as well as 
Corollary 2.5 and Theorem 2.6 of \cite{arguin-zindy}.
\begin{theorem}
\label{thm: PD}
Let $\beta>\beta_c$ and $\xi=(\xi_k,k\in\N)$ be a Poisson-Dirichlet variable of parameter $\beta_c/\beta$.
Denote by $E$ the expectation with respect to $\xi$.
For any continuous function $F: [0,1]^{\frac{s(s-1)}{2}}\to \R$ of the overlaps of $s$ replicas:
$$
\lim_{N\to\infty} 
\e \left[  {\mathcal G}_{\beta, N}^{\times s} \left( F(q_{ll'}) \right)\right] 
=
E \left[  \sum_{k_1\in\N,...,k_s\in\N} \xi_{k_1}\dots \xi_{k_s} ~ F(\delta_{k_lk_{l'}}) \right].
$$
\end{theorem}
The above is one of the few rigorous results known on the Gibbs measure of log-correlated fields at low temperature. 
Theorem \ref{thm: PD} is a step closer to the conjecture of Duplantier, Rhodes, Sheffield \& Vargas (see Conjecture 11 in \cite{drsv} and Conjecture 6.3 in \cite{rhodes-vargas}) 
that the Gibbs measure, as a random probability measure on $V_N$, should be atomic in the limit with the size of the atoms being Poisson-Dirichlet. 
Theorem \ref{thm: PD} falls short of the full conjecture because only test-functions of the overlaps are considered.
Finally, it is expected that the Poisson-Dirichlet statistics emerging here is related to the Poissonian statistics of the thinned extrema of the 2D GFF proved by Biskup \& Louidor in \cite{biskup-louidor} based on the convergence of the maximum established by Bramson, Ding \& Zeitouni \cite{bramson-ding-zeitouni}.
To recover the Gibbs measure from the extremal process, some properties of the cluster of points near the maxima must be known.

The rest of this paper is dedicated to the proof of Theorem \ref{thm: overlap}.
In Section \ref{sect: generalized}, a generalized version of the GFF (whose variance is scale-dependent) is introduced. 
It is a kind of non-hierarchical GREM and is related to a model studied by Fyodorov \& Bouchaud in \cite{fyodorov-bouchaud}.
The proof of Theorem \ref{thm: overlap} is given in Section \ref{sect: overlap}. It relates the overlap distribution of the 2D GFF to the free energy of the generalized GFF.
The free energy of the generalized GFF needed in the proof is computed in Section \ref{sect: free energy}.

%%%%%%%%%%%%%%%%%%%%%%%%%%%%%%%%%%%%%%%%%%%%%%%%%%%%%%%%%%%%%%%%%%%%%%%%%%%%%%

\section{The multiscale decomposition and a generalized GFF}
\label{sect: generalized}
In this section, we construct a Gaussian field from the GFF whose variance is scale-dependent.
The construction uses a multiscale decomposition along each vertex. 
The construction is analogous to a {\it Generalized Random Energy Model} of Derrida \cite{derrida}, but where correlations are non-hierarchical.
Here, only two different values of the variance will be needed though the construction can be directly generalized to any finite number of values.

Consider $0<t<1$.
We assume to simplify the notation that $N^{1-t}$ is an even integer and that $N^{t}$ divides $N$.
The case of general $t$'s can also be done by making trivial corrections along the construction.

For $v\in V_N$, we write $[v]_t$ for the unique box with $N^{1-t}$ points on each side and centered at $v$. 
If $[v]_t$ is not entirely contained in $V_N$, we take the convention that $[v]_t$ is the intersection of the square box with $V_N$.  For $t=1$, take $[v]_1=v$.
The $\sigma$-algebra  $\F_{[v]_t^c}$ is the $\sigma$-algebra generated by the field outside $[v]_t$. 
We define
$$
\phi_{[v]_t}:= \E\left[ \phi_v ~\big| ~ \F_{[v]_t^c} \right]=\E\left[ \phi_v ~\big| ~ \F_{\partial [v]_t}\right]\ ,
$$
where the second equality holds by the Markov property of the Gaussian free field, see Lemma \ref{lem: GFF}.
Clearly, for any $v\in V_N$, the random variable $\phi_{[v]_t}$ is Gaussian.
Moreover, by Lemma \ref{lem: GFF},
\begin{equation}
\label{eqn: field decomp}
\phi_{[v]_t}=\sum_{u\in \partial [v]_t} p_{t,v}(u) \phi_u \ ,
\end{equation}
where $p_{t,v}(u)=P_v(S_{\tau_{[v]_t}}=u)$ is the probability that a simple random walk starting at $v$ hits $u$ at the first exit time of $[v]_t$.

The following {\it multiscale decomposition} holds trivially
\begin{equation}
\phi_v= \phi_{[v]_t} + \left(\phi_v-  \phi_{[v]_t} \right)\ .
\end{equation}
The decomposition suggests the following scale-dependent perturbation of the field. 
For $0<\alpha<1$ and
$\vec{\sigma}=(\sigma_1, \sigma_2) \in \R_+^2$, consider for $v\in V_N$,
\begin{equation}
\label{eqn: psi}
\psi_v:=\sigma_1 \phi_{[v]_\alpha} + \sigma_2 \left(\phi_v-  \phi_{[v]_\alpha} \right)\ .
\end{equation}
The Gaussian field  $(\psi_v, v\in V_N)$ will be called the $(\alpha,\vec{\sigma})$-GFF on $V_N$. \\

To control the boundary effects, it is necessary to consider the field in a box slightly smaller than $V_N$.
 For $\rho \in (0,1)$, let
\begin{equation}
\label{eqn:Arho}
A_{N,\rho}:= \{v \in V_N : d_1(v,\partial V_N) \ge N^{1-\rho}\}\ ,
\end{equation}
 where $d_1(v,B):=\inf \{ \| v-u\| \, ; \, u \in B\}$ for any set $B \subset \Z^2.$
We always take $\rho<\alpha$ so that $[v]_\alpha$ is contained in $V_N$ for any $v\in A_{N,\rho}$.
We write ${\mathcal G}_{\beta, N,\rho}^{(\alpha,\vec{\sigma})}(\cdot)$ for the Gibbs measure of $(\alpha,\vec{\sigma})$-GFF restricted to $A_{N,\rho}$
\begin{equation*}
{\mathcal G}_{\beta, N,\rho}^{(\alpha,\vec{\sigma})}(\{v\}):=\frac{\ee^{\beta \psi_v}}{Z_{N,\rho}^{(\alpha,\vec{\sigma})}(\beta)}, \qquad v \in  A_{N,\rho} ,
\end{equation*}
where $Z_{N,\rho}^{(\alpha,\vec{\sigma})}(\beta):= \sum_{v\in A_{N,\rho}} \ee^{\beta \psi_v}\ .$

The associated free energy is given by
$$
f_{N,\rho}^{(\alpha,\vec{\sigma})}(\beta):= \frac{1}{\log N^2} \log Z_{N,\rho}^{(\alpha,\vec{\sigma})}(\beta), \qquad \forall \beta >0.
$$
(Note that $\log \#A_{N,\rho}=(1+o_N(1))\log N^2$.)
Its $L_1$-limit is a central quantity needed to apply Bovier-Kurkova technique. 
This limit is better expressed in terms of the free energy of the REM model consisting of $N^2$ i.i.d. Gaussian variables of variance $\frac{\sigma^2}{\pi} \log N^2$:
\begin{equation}
\label{eqn: rem free}
f(\beta; \sigma^2):=
\begin{cases}
1+\frac{\beta^2 \sigma^2}{2 \pi}, &\text{ if $\beta\leq \beta_c(\sigma^2):= \frac{\sqrt{2 \pi}}{{\sigma}},$}\\
\sqrt{\frac{2}{\pi}}\sigma \beta,  &\text{ if $\beta\geq \beta_c(\sigma^2).$}
\end{cases}
\end{equation}

\begin{theorem}
\label{thm:freeenergyperturbed}
Fix $\alpha \in (0,1)$ and $\vec{\sigma}= (\sigma_1,\sigma_2) \in \R_+^2$ and let  
$V_{12}:=\sigma_1^2\alpha+\sigma_2^2(1-\alpha)$. Then, for any $\rho < \alpha,$ and for all $\beta>0$
\begin{equation}
\label{eqn: free energy pert}
\lim_{N\to\infty}f_{N,\rho}^{(\alpha,\vec{\sigma})}(\beta)=f^{(\alpha,\vec{\sigma})}(\beta):=
\begin{cases}
 f(\beta; V_{12}), \qquad &\text{ if $\sigma_1\leq \sigma_2$,}\\
 \alpha f(\beta; \sigma_1^2) + (1-\alpha)f(\beta; \sigma_2^2),  \qquad &\text{ if $\sigma_1\geq \sigma_2$,}
\end{cases}
\end{equation}
where the convergence holds almost surely and in $L^1$.
\end{theorem}
Note that the limit does not depend on $\rho$.

%%%%%%%%%%%%%%%%%%%%%%%%%%%%%%%%%%%%%%%%%%%%%%%%%%%%%%%%%%%%%%%%%%

\section{Proof of Theorem \ref{thm: overlap}}
\label{sect: overlap}

\subsection{The Gibbs measure close to the boundary}

The first step in the proof of Theorem \ref{thm: overlap} is to show that points close to the boundary do not carry any weight in the Gibbs measure of the GFF in $V_N$.
The result would not necessarily hold if we considered instead the outside of $V_N^\delta$ which is much larger than the outside of $A_{N,\rho}$.
\begin{lemma}
\label{lem: boundary}
For any $\rho>0$, 
\begin{equation}
\lim_{N\to\infty}\mathcal G_{\beta, N}(A^c_{N,\rho}) =0, \qquad \text{ in $\p$-probability.}
\end{equation}
\end{lemma}
Before turning to the proof, we claim that the lemma implies that, for any $r\in [0,1]$ and $\rho\in(0,1)$,
\begin{equation}
\label{eqn: x rho}
\lim_{N\to\infty}\big|x_{\beta,N} (r)-x_{\beta,N,\rho} (r)\big|=0\ ,
\end{equation}
where 
\begin{equation}
 x_{\beta,N,\rho}(r ):= \E\mathcal G_{\beta,N,\rho}^{\times 2}\{q(v,v')\leq r\}, \qquad \text{ $r\in [0,1]$}\ .
\end{equation}
is the two-overlap distribution of the Gibbs measure of the GFF $(\phi_v,v\in V_N)$ restricted to $A_{N,\rho}$
\begin{equation*}
{\mathcal G}_{\beta, N, \rho}(\{v\}):=\frac{\ee^{\beta \phi_v}}{Z_{N,\rho}(\beta)}, \qquad v \in  A_{N,\rho} ,
\end{equation*}
for $Z_{N,\rho}(\beta):= \sum_{v\in A_{N,\rho}} \ee^{\beta \phi_v}$.
Indeed, introducing an auxiliary term
$$
\begin{aligned}
\big|x_{\beta,N} (r)-x_{\beta,N,\rho} (r)\big|&\leq
\big| \E\mathcal G_{\beta,N}^{\times 2}\big\{q(v,v')\leq r\big\} - \E\mathcal G_{\beta,N}^{\times 2}\big\{q(v,v')\leq r; v,v'\in A_{N,\rho}\big\}\big|\\
&+
\big|  \E\mathcal G_{\beta,N}^{\times 2}\big\{q(v,v')\leq r; v,v'\in A_{N,\rho}\big\}-\E\mathcal G_{\beta,N,\rho}^{\times 2}\big\{q(v,v')\leq r\big\} \big|\ .
\end{aligned}
$$
The first term is smaller than $2\ \E\mathcal G_{\beta,N}(A_{N,\rho}^c)$.
The second term equals
$$
\begin{aligned}
&\E\mathcal G_{\beta,N,\rho}^{\times 2}\big\{q(v,v')\leq r\big\}- \E\mathcal G_{\beta,N}^{\times 2}\big\{q(v,v')\leq r; v,v'\in A_{N,\rho}\big\}\\
&\hspace{4cm}=
\E\left[\frac{\mathcal G_{\beta,N}^{\times 2}\big\{q(v,v')\leq r; v,v'\in A_{N,\rho}\big\}}{\mathcal G_{\beta,N}^{\times 2}\big\{ v,v'\in A_{N,\rho}\big\}}
\left(1-\mathcal G_{\beta,N}^{\times 2}\big\{ v,v'\in A_{N,\rho} \big\} \right)\right],
\end{aligned}
$$
which is also smaller than $2\ \E\mathcal G_{\beta,N}(A_{N,\rho}^c)$. 
Lemma \ref{lem: boundary} then implies \eqref{eqn: x rho} as claimed.
%Note that Equation \eqref{eqn: x rho} reduces the proof of Theorem \ref{thm: overlap} to computing the limit of $x_{\beta,N,\rho}$ for some $\rho$.

\begin{proof}[Proof of Lemma \ref{lem: boundary}]
Let $\epsilon>0$ and $\lambda>0$. The probability can be split as follows
$$
\begin{aligned}
\p\left(\mathcal G_{\beta, N}(A^c_{N,\rho}) > \epsilon \right)
&\leq \p\left(\mathcal G_{\beta, N}(A^c_{N,\rho}) > \epsilon, \left|\frac{1}{\log N^2}\log Z_N(\beta) - f(\beta)\right|\leq \lambda \right)\\\
&+
 \p\left( \left|\frac{1}{\log N^2}\log Z_N(\beta) - f(\beta) \right|> \lambda \right)\ ,
\end{aligned}
$$
where $f(\beta)$ is defined in \eqref{eqn: free energy GFF}.
The second term converges to zero by \eqref{eqn: free energy GFF}. 
The first term is smaller than
\begin{equation}
\label{eqn: self-overlap}
\begin{aligned}
  \p\left( \frac{1}{\log N^2}\log \sum_{v\in A^c_{N,\rho}} \exp\beta \phi_v > f(\beta)-\lambda+\frac{\log \epsilon}{\log N^2}\right) .
  \end{aligned}
\end{equation}
Since the free energy is a Lipschitz function of the variables $\phi_v$, see e.g. Theorem 2.2.4 in \cite{talagrand}, the free energy self-averages, that is for any $t>0$
 $$
 \lim_{N\to\infty}\p\left( \left|\frac{1}{\log N^2}\log \sum_{v\in A_{N,\rho}^c} \exp\beta \phi_v -\frac{1}{\log N^2} \E\left[\log \sum_{v\in A_{N,\rho}^c} \exp\beta \phi_v\right]\right|\geq t\right)= 0\ .
 $$
To conclude the proof, it remains to show that for some $C<1$ (independent of $N$ but dependent on $\rho$)
 \begin{equation}
 \label{eqn: expect bound}
 \limsup_{N\to\infty} \frac{1}{\log N^2} \E\left[\log \sum_{v\in A_{N,\rho}^c} \exp\beta \phi_v\right] < C f(\beta).
 \end{equation}
Note that by Lemma \ref{lem: GFF}, the maximal variance of $\phi_v$ in $V_N$ is $ \frac{1}{\pi}\log N^{2}+O_N(1)$.
 Pick $(g_v,v\in A_{N,\rho}^c)$ independent centered Gaussians (and independent of $(\phi_v)_{v \in A_{N,\rho}^c}$) with variance given by $\E[g_v^2]=  \frac{1}{\pi}\log N^{2}+O_N(1) - \E[\phi_v^2]$.
 Jensen's inequality applied to the Gibbs measure implies that $\E[\log \sum_{v\in A_{N,\rho}^c} \exp \beta (\phi_v+g_v)]\geq \E[\log \sum_{v\in A_{N,\rho}^c} \exp \beta \phi_v]$.
 Moreover, by a standard comparison argument (see Lemma \ref{lem: slepian} in the Appendix), $\E[\log \sum_{v\in A_{N,\rho}^c} \exp \beta (\phi_v+g_v)]$ is smaller than the expectation for i.i.d.~variables with identical variances. The two last observations imply that
 $$
  \frac{1}{\log N^2} \E\left[\log \sum_{v\in A_{N,\rho}^c} \exp\beta \phi_v\right] \leq  \frac{1}{\log N^2} \E\left[\log \sum_{v\in A_{N,\rho}^c}\exp\beta \widetilde\phi_{ v} \right],
 $$
 where $(\widetilde \phi_v, v\in A_{N,\rho}^c)$ are i.i.d.~centered Gaussians of variance $\frac{1}{\pi}\log N^{2}+O_N(1)$.
  Since $\#A_{N,\rho}^c=N^2-|A_{N,\rho}|= 4N^{2-\rho}(1+o_N(1))$, the free energy of these i.i.d.~Gaussians in the limit $N\to\infty$ is given by \eqref{eqn: rem free}
 $$
 \lim_{N\to\infty} \frac{1}{\log 4N^{2-\rho}}\E\left[\log \sum_{ v\in A_{N,\rho}^{ c}} \exp\beta \widetilde\phi_{ v}\right] = 
 \begin{cases}
 1+ \frac{\beta^2}{2\pi}\left(1-\frac{\rho}{2}\right){ ^{-1}}, & \text{ $\beta< \sqrt{2\pi}\left(1-\frac{\rho}{2}\right)^{1/2}$},\\
 \sqrt{\frac{2}{\pi}} \left(1-\frac{\rho}{2}\right)^{-1/2}  \beta,  &\text{ $\beta\geq  \sqrt{2\pi}\left(1-\frac{\rho}{2}\right)^{1/2}$}\ .
 \end{cases}
 $$
 The last two equations then imply
 $$
  \limsup_{N\to\infty} \frac{1}{\log N^2} \E\left[\log \sum_{v\in A_{N,\rho}^{{ c}}} \exp\beta \phi_v\right] \leq 
   \begin{cases}
 \left(1-\frac{\rho}{2}\right) + \frac{\beta^2}{2\pi}, & \text{ $\beta< \sqrt{2\pi}\left(1-\frac{\rho}{2}\right)^{1/2}$},\\
 \sqrt{\frac{2}{\pi}} \left(1-\frac{\rho}{2}\right)^{1/2}  \beta,  &\text{ $\beta\geq  \sqrt{2\pi}\left(1-\frac{\rho}{2}\right)^{1/2}$}\ .
 \end{cases}
   $$
 It is then straightforward to check that, for every $\beta$, the right side is strictly smaller than $f(\beta)$ as claimed.
 \end{proof}

\subsection{An adaptation of the Bovier-Kurkova technique}
Theorem \ref{thm: overlap} follows from Equation \eqref{eqn: x rho} and
\begin{proposition}
For $\beta >\beta_c= \sqrt{2\pi}$,
$$
\lim_{\rho \to 0}\lim_{N\to\infty} x_{\beta,N,\rho} (r) 
=
\begin{cases}
\frac{\beta_c}{\beta}, &\text{ for $0 \le r <1$,}\\
1, &\text{ for $r=1$.}
\end{cases}
$$
\end{proposition}

\begin{proof}
Without loss of generality, we suppose that  $\lim_{\rho\to 0} \lim_{N\to\infty}x_{\beta, N,\rho}=x_\beta$ 
in the sense of weak convergence. Uniqueness of the limit $x_\beta$ will then ensure the convergence for the whole sequence by compactness.
Note also that by right-continuity and monotonicity of $x_\beta$, it suffices to show
\begin{equation}
\label{eqn: to prove}
\int_\alpha^1   x_\beta (r) dr = \frac{\beta_c}{\beta} (1-\alpha) , \qquad \text{ for a dense set of $\alpha$'s in $[0,1]$.}
\end{equation}
We can choose a dense set of $\alpha$ such that none of them are atoms of $x_\beta$, that is $x_\beta(\alpha)-x_\beta(\alpha^-)=0$.

Now recall Theorem \ref{thm:freeenergyperturbed}. Pick $\vec{\sigma}=(1,1+u)$ for some parameter $|u|\leq 1$.
 Since $\beta>\sqrt{2\pi}$, $u$ can be taken small enough so that $\beta$ is larger than the critical $\beta$'s of the limit.
The goal is to establish the following equality:
\begin{equation}
\label{eqn: equality}
\int_\alpha^1   x_\beta (r) dr =\lim_{\rho\to 0} \lim_{N\to\infty} \frac{\pi}{\beta^2} \frac{\partial}{\partial u} f_{N,\rho}^{( \alpha, \vec{\sigma})}{ (\beta)} \Big|_{u=0}\ .
\end{equation}
The conclusion follows from this equality.
Indeed, by construction, the function $u\mapsto {  f_{N,\rho}^{( \alpha, \vec{\sigma})}(\beta)}$ is convex. In particular, the limit of the derivatives is the derivative of the limit
at any point of differentiability. Therefore, a straightforward calculation from \eqref{eqn: free energy pert} with $\sigma_1=1$ and $\sigma_2=1+u$ gives:
 \begin{equation}
\lim_{N\to\infty} \frac{\pi}{\beta^2} \frac{\partial}{\partial u} f_{N,\rho}^{( \alpha, \vec{\sigma})}(\beta) =
\begin{cases}
\frac{\sqrt{2\pi}}{\beta} \frac{(1-\alpha)(1+u)}{\sqrt{\alpha + (1-\alpha)(1+u)^2}}, & \text{ if $u>0$,}\\
\frac{\sqrt{2\pi}}{ \beta} (1-\alpha), & \text{ if $u<0$.}
\end{cases}
\end{equation}
This gives \eqref{eqn: to prove} at $u=0$.

We introduce the notation for the {\it overlap at scale $\alpha$}:
\begin{equation}
\label{eqn: q alpha}
q_\alpha(v,v'):=\frac{1}{\frac{1}{\pi}\log N^2}\E\left[\left(\phi_v-  \phi_{[v]_\alpha} \right)\left(\phi_{v'}-  \phi_{[v']_\alpha} \right)\right],
\end{equation}
Equality \eqref{eqn: equality} is proved via two identities: 
\begin{eqnarray}
\label{eqn: BK1}
\int_\alpha^1   x_{\beta, N,\rho} (r) dr &= &(1-\alpha)- \E \mathcal G_{\beta, N, \rho}^{\times 2} \big[ q(v,v') -\alpha ; q(v,v')\geq \alpha\big],\\
\label{eqn: BK2}
\frac{\pi}{\beta^2} \frac{\partial}{\partial u}{f_{N,\rho}^{( \alpha, \vec{\sigma})}}(\beta) \Big|_{u=0} &=&\E \mathcal G_{\beta, N, \rho} \big[ q_\alpha(v,v) \big] -
 \E \mathcal G_{\beta, N, \rho}^{\times 2} \big[q_\alpha(v,v')  ; v'\in [v]_\alpha\big]\ .
\end{eqnarray}
The first identity holds since by Fubini's theorem
$$
\begin{aligned}
\int_\alpha^1   x_{\beta, N,\rho}  (r) dr&= 
\E \mathcal G_{\beta,N,\rho}^{\times 2} \left[ \int_\alpha^1 1_{\{ r\geq q(v,v')\}} dr\right]\\
&= \E \mathcal G_{\beta,N,\rho}^{\times 2} \big[ 1-\alpha;   q(v,v') <\alpha \big] +  \E \mathcal G_{\beta,N,\rho}^{\times 2} \big[ 1-q(v,v');   q(v,v') \geq\alpha \big]\ .
\end{aligned}
$$

For the second identity, direct differentiation gives
$$
\frac{\pi}{\beta^2} \frac{\partial}{\partial u} {f_{N,\rho}^{( \alpha, \vec{\sigma})}}(\beta) \Big|_{u=0}=\frac{1}{\frac{1}{\pi}\log N^2}\E\mathcal G_{\beta,N,\rho}\big[ \phi_v-\phi_{[v]_\alpha}\big]\ .
$$
The identity is then obtained by Gaussian integration by parts.

To prove \eqref{eqn: equality}, we need to relate the overlap at scale $\alpha$ with the overlap as well as the event $\{q(v,v')\geq \alpha\}$ with the event $\{v'\in[v]_\alpha\}$.
This is slightly complicated by the boundary effect present in GFF. 
The equality in the limit $N\to\infty$ between the first terms of \eqref{eqn: BK1} and \eqref{eqn: BK2} is easy.
Because $(\phi_u- \E[\phi_u | \F_{[v]_\alpha^c}], u \in [v]_\alpha)$ has the law of a GFF in $[v]_\alpha$, it follows from Lemma \ref{lem: green estimate} that
$$
 \E\big[(\phi_v-\phi_{[v]_\alpha})^2\big]=\frac{(1-\alpha)}{\pi} \log N^{2} + O_N(1)\ .
$$
Therefore, we have for $v\in A_{N,\rho}$
$$
 \lim_{N\to\infty} \E \mathcal G_{\beta, N, \rho} \left[ q_\alpha(v,v) \right] = 1-\alpha \ .
$$
It remains to establish the equality between the second terms of \eqref{eqn: BK1} and \eqref{eqn: BK2}. Here, a control of the boundary effect is necessary.
The following observation is useful to relate the overlaps and the distances: if $v,v'\in A_{N,\rho}$, Lemma \ref{lem: green estimate} gives
\begin{equation}
\label{eqn: q in Arho}
1-\rho-\frac{\log \|v-v'\|^2}{\log N^2} + o_N(1)\leq q(v,v') \leq 1-\frac{\log \|v-v'\|^2}{\log N^2} + o_N(1)\ .
\end{equation}
On one hand, the right inequality proves the following implication 
\begin{equation}
\label{eqn: right}
\text{$q(v,v')\geq \alpha+\varepsilon$ for some $\varepsilon>0$} \Longrightarrow  \|v-v'\|^{ 2} \leq c N^{2(1-\alpha-\varepsilon)}\ ,
\end{equation}
for some constant $c$ independent of $N$ and $\rho$.
On the other hand, the left inequality gives: 
\begin{equation}
\label{eqn: left}
v'\in [v]_\alpha \Longrightarrow  \text{$q(v,v')\geq \alpha-2\rho$.}
\end{equation} 

Using this, we show
\begin{equation}
\begin{aligned}
\label{eqn: show1}
\Delta_1(N,\rho)&:= \Big|\E \mathcal G_{\beta, N, \rho}^{\times 2} \left[ q(v,v') -\alpha ; q(v,v')\geq \alpha\right]-\E \mathcal G_{\beta, N, \rho}^{\times 2} \left[ q_\alpha(v,v')  ; q(v,v')\geq \alpha\right]\Big| \to 0\ ,\\
\Delta_2(N,\rho)&:=\Big|\E \mathcal G_{\beta, N, \rho}^{\times 2} \left[ q_\alpha(v,v')  ; q(v,v')\geq \alpha\right]-\E \mathcal G_{\beta, N, \rho}^{\times 2} \left[ q_\alpha(v,v')  ; v'\in[v]_\alpha\right]\Big| \to 0\ ,
\end{aligned}
\end{equation}
 in the limit $N\to\infty$ and $\rho\to 0$.
 Let  $\varepsilon>0$. Remark that
\begin{equation}
\label{eqn: rem}
0\leq \E \mathcal G_{\beta, N, \rho}^{\times 2} \left[ q(v,v') -\alpha ; q(v,v')\geq \alpha\right]- \E \mathcal G_{\beta, N, \rho}^{\times 2} \left[ q(v,v') -\alpha ; q(v,v')\geq \alpha+\varepsilon\right]\leq \varepsilon\ .
\end{equation}
To establish the equality of the overlaps on the event $\{q(v,v')\geq \alpha+\varepsilon\}$, consider the decomposition,
\begin{equation}
\label{eqn: fusion1}
\begin{aligned}
&\E\left[\left(\phi_v-  \phi_{[v]_\alpha} \right)\left(\phi_{v'}-  \phi_{[v']_\alpha} \right)\right]=\\
&\E\left[\left(\phi_v-  \E[\phi_{ v}|\F_{[v']^c_\alpha}] \right)\left(\phi_{v'}-  \phi_{[v']_\alpha} \right)\right]
+
\E\left[\left(\E[\phi_{v}|\F_{[v']^c_\alpha}]-  \phi_{[v]_\alpha} \right)\left(\phi_{v'}-  \phi_{[v']_\alpha} \right)\right].
\end{aligned}
\end{equation}
On the event $\{q(v,v')\geq \alpha+\varepsilon\}$, \eqref{eqn: right} implies $\|v-v'\|^{ 2} \leq  c N^{2(1-\alpha-\varepsilon)}$. 
Therefore, the first term of the right side of \eqref{eqn: fusion1} is by Lemma \ref{lem: green estimate}
\begin{equation}
\label{eqn: 1stterm}
\E\left[\left(\phi_v-  \E[\phi_{ v}|\F_{[v']^c_\alpha} \right)\left(\phi_{v'}-  \phi_{[v']_\alpha} \right)\right]=\frac{ 2}{\pi}\log \frac{ N^{(1-\alpha)}}{\|v-v'\|} + O_N(1)\ .
\end{equation}
The second term is negligible.
Indeed, by Cauchy-Schwarz inequality, it suffices to prove that
\begin{equation}
\label{eqn: fusion2}
\E\left[\left(\E[\phi_{v}|\F_{[v']^c_\alpha}]-  \phi_{[v]_\alpha} \right)^2\right]=O_N(1)\ .
\end{equation}
For this, write $\widetilde B$ for the box $[v]_\alpha\cap [v']_\alpha$. We have 
$$
\phi_v -\phi_{[v]_\alpha} =(\phi_v-\E[\phi_v| \F_{\widetilde B^c}])+(\E[\phi_v| \F_{\widetilde B^c}]-\phi_{[v]_\alpha})\ .
$$
Since $\phi_v-\E[\phi_v| \F_{\widetilde B^c}]$ is independent of $\F_{\widetilde B^c}$ and $\E[\phi_v| \F_{\widetilde B^c}]-\phi_{[v]_\alpha}$ is $\F_{\widetilde B^c}$-measurable  (observe that $\F_{\widetilde B^c}\supset \F_{[v]_\alpha^c}$), we get 
$$
 \E[(\phi_v-\phi_{[v]_\alpha})^2]= \E[(\phi_v-\E[\phi_v| \F_{\widetilde B^c}])^2]+ \E[(\E[\phi_v| \F_{\widetilde B^c}]-\phi_{[v]_\alpha})^2]\ .
$$ 
 Moreover, $\E[(\phi_v-\E[\phi_v| \F_{\widetilde B^c}])^2]$ and $\E[(\phi_v-\phi_{[v]_\alpha})^2]$
are both equal to $\frac{1-\alpha}{\pi}\log N^2  + O_N(1)$ by Lemma \ref{lem: green estimate} and the fact that distances of $v$ to vertices in $\partial\widetilde B$ and $\partial[v]_\alpha$ are both proportional to $N^{1-\alpha}$. Therefore $\E[(\E[\phi_v| \F_{\widetilde B^c}]-\phi_{[v]_\alpha})^2]=O_N(1)$. The same argument with $\phi_{[v]_\alpha}$ replaced by $\E[\phi_v|\F_{[v']_\alpha^c}]$ shows that  $\E[(\E[\phi_v| \F_{\widetilde B^c}]-\E[\phi_v|\F_{[v']_\alpha^c}])^2]=O_N(1)$. The two equalities imply \eqref{eqn: fusion2}.
Equations \eqref{eqn: 1stterm} and \eqref{eqn: fusion2} give
\begin{equation}
\label{eqn: qalpha}
q_\alpha(v,v')= 1-\alpha-\frac{\log \|v-v'\|^2}{\log N^2} + o_N(1), \qquad \text{  on $\{q(v,v')\geq \alpha+\varepsilon\}$.}
\end{equation}
Equations \eqref{eqn: q in Arho},  \eqref{eqn: rem} and \eqref{eqn: qalpha} yield $\Delta_1(N,\rho)\to 0$ in the limit $N\to\infty$, $\rho\to 0$ and $\varepsilon\to 0$. 

For $\Delta_2(N,\rho)$, let $\varepsilon'>2\rho$. 
For $v'\in [v]_\alpha$, \eqref{eqn: left} implies $q(v,v')\geq \alpha - 2\rho$. 
On the other hand, by \eqref{eqn: right}, $q(v,v')\geq \alpha+\varepsilon'$ implies $v'\in[v]_\alpha$.
These two observations give the estimate
$$
\Delta_2(N,\rho)
\leq \E \mathcal G_{\beta, N, \rho}^{\times 2} \big[ q_\alpha(v,v')  ; q(v,v')\in [\alpha-\varepsilon',\alpha+\varepsilon']\big]\ .
$$
The right side is clearly smaller than
$$
x_{\beta,N,\rho}(\alpha+\varepsilon')-x_{\beta,N,\rho}(\alpha-\varepsilon')\ .
$$
Under the successive limits $N\to\infty$, $\rho\to 0$, then $\varepsilon'\to 0$, the right side becomes $x_\beta(\alpha)-x_\beta(\alpha-)$.
This is zero since $\alpha$ was chosen not to be an atom of $x_\beta$.
\end{proof}

%%%%%%%%%%%%%%%%%%%%%%%%%%%%%%%%%%%%%%%%%%%%%%%%%%%%%%%%%%%%%%%%%

\section{The free energy of the $(\alpha,\vec{\sigma})$-GFF: proof of Theorem \ref{thm:freeenergyperturbed}}
\label{sect: free energy}

The computation of the free energy of the $(\alpha,\vec{\sigma})$-GFF is divided in two steps. 
First, an upper bound is found by comparing the field $\psi$ in $A_{N,\rho}$ with a ``non-homogeneous'' GREM having the same free energy as a standard 2-level GREM.
Second, we get a matching lower bound using the trivial inequality $f_{N,\rho}^{(\alpha,\vec{\sigma})}(\beta) \ge \ \frac{1}{\log N^2}  \log  \sum_{v \in V_N^\delta} \ee^{\beta \psi_v} $.
The limit of the right term is computed following the method of Daviaud \cite{daviaud}.

\subsection{Proof of the upper bound} 
\label{sect: upperbound}
For conciseness, we only prove the case $\sigma_1\geq \sigma_2$, by a comparison argument with a $2$-level GREM. 
The case $\sigma_1\leq \sigma_2$ is done similarly by comparing with a REM.
The comparison argument will have to be done in two steps to account for boundary effects. 

Divide the set $A_{N,\rho}$ into square boxes of side-length $N^{1-\alpha}/100$. (The factor $1/100$ is a choice. We simply need these boxes to be smaller than the neighborhoods $[v]_\alpha$, yet of the same order of length in $N$.) Pick the boxes in such a way that each $v\in A_{N,\rho}$ belongs to one and only one of these boxes.
The collection of boxes is denoted by $\mathcal B_\alpha$ and $\partial  \mathcal B_\alpha$ denotes $\bigcup_{B\in \mathcal B_\alpha} \partial B$.
For $v\in  A_{N,\rho}$, we write $B(v)$ for the box of  $\mathcal B_\alpha$ to which $v$ belongs.
For  $B\in \mathcal B_{ \alpha}$, denote by $\widetilde B\supset B$ the square box given by the intersections
of all $[u]_\alpha$, $u\in B$, see figure \ref{fig: boxes}. Remark that the side-length of $\widetilde B$ is $cN^{1-\alpha}$, for some constant $c$.
For short, write $\phi_{\widetilde B}:=\E[\phi_{v_B}|\F_{\widetilde B^c}]$ where $v_B$ is the center of the box $B$.
The idea in constructing the GREM is to associate to each point $v  \in B$ the same contribution at scale $\alpha$, namely $\phi_{\widetilde B}$.
One problem is that $\phi_{\widetilde B}$ will not have the same variance for every $B$ since it depends on the distance to the boundary. 
This is the reason why the comparison will need to be done in two steps.
\begin{figure}[h] 
\begin{center}
\includegraphics[height=4cm]{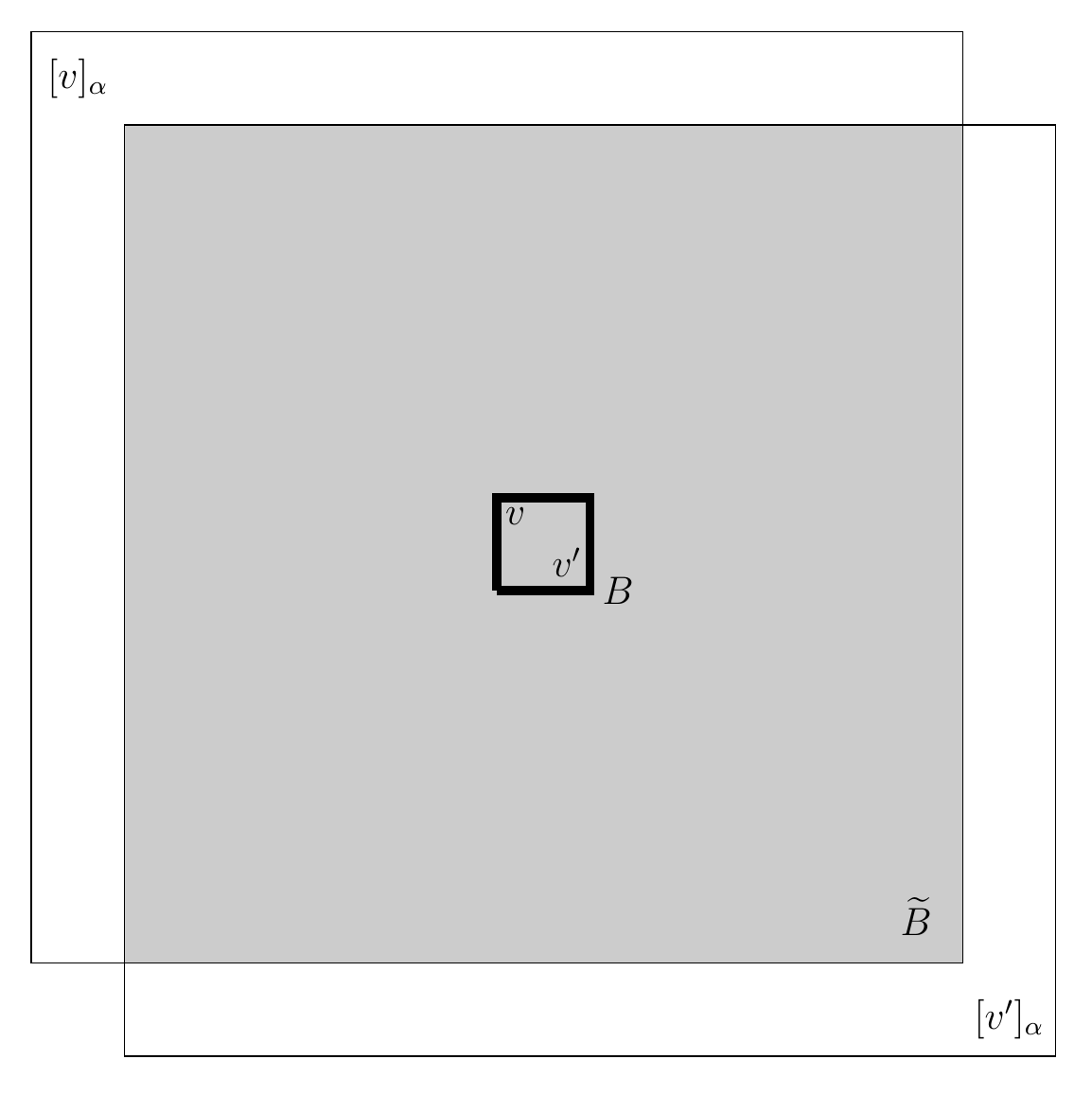}
\end{center}
\caption{The box $B\in \mathcal B_\alpha$ and the corresponding box $\widetilde B$ which is the intersection of all the neighborhoods $[v]_\alpha$, $v\in B$.}
\label{fig: boxes}
\end{figure}

First, consider the hierarchical Gaussian field $(\widetilde\psi_v, v\in A_{N,\rho})$:  
\begin{equation}
\label{eqn: tilde psi}
\widetilde \psi_v= g^{(1)}_{B(v)} + g^{(2)}_v,
\end{equation}
where $(g^{(2)}_v, v\in A_{N,\rho} )$ are independent centered Gaussians { (also independent from $(g^{(1)}_B, B\in \mathcal B_\alpha)$)} with variance
$$
\E[(g^{(2)}_v)^2]= \E[\psi_v^2] - \E[ (g^{(1)}_{B(v)})^2]\ .
$$
This ensures that $\E[\psi_v^2]= \E[\widetilde\psi_v^2]$ for all $v\in  A_{N,\rho}$.
The variables $(g^{(1)}_B, B\in \mathcal B_\alpha)$  are also independent centered Gaussians with  variance chosen to be $\sigma_1^2\E[\phi_{\widetilde B}^2]+C$ for some constant $C\in \R$ independent of $B$ in $ \mathcal B_{ \alpha}$ and independent of $N$.
The next lemma ensures that
\begin{equation}
\label{eqn: corr bound}
\E[\psi_v\psi_{v'}]\geq \E[\widetilde\psi_v\widetilde \psi_{v'}]\ .
\end{equation}
\begin{lemma}
\label{lem: correlations}
Consider the field $(\psi_v, v\in  A_{N,\rho})$ as in \eqref{eqn: psi}. Then $\E[\psi_v\psi_{v'}]\geq 0$.
Moreover, 
if $v$ and $v'$ both belong to $B\in \mathcal B_\alpha$, then
$$
\E[\psi_v\psi_{v'}]\geq \sigma_1^2\E[\phi_{\widetilde B}^2]+C\ ,
$$
for some constant $C\in\R$ independent of $N$.
\end{lemma}

\begin{proof}
For the first assertion, write
$$
\psi_v=(\sigma_1-\sigma_2) \phi_{[v]_\alpha}+\sigma_2 \phi_v\ .
$$
The representation $\phi_{[v]_\alpha}=\sum_{u\in \partial [v]_\alpha} p_{\alpha,v}(u)~\phi_u$ of Lemma \ref{lem: GFF} and the fact that $\sigma_1>\sigma_2$ 
imply that $\E[\psi_v\psi_v'] \ge 0$ since the field $\phi$ is positively correlated by \eqref{eqn: cov}.  

Suppose now that $v,v'\in B$ where $B\in \mathcal B_\alpha$. The covariance can be written as
\begin{equation}
\label{eqn: covariance psi}
\begin{aligned}
\E[\psi_v\psi_{v'}]&= \sigma_1^2\E\left[\phi_{[v]_\alpha}\phi_{[v']_\alpha}\right] + \sigma_2^2\E\left[(\phi_v-\phi_{[v]_\alpha})(\phi_{v'}-\phi_{[v']_\alpha})\right] \\
&+ \sigma_1\sigma_2 \E\left[\phi_{[v]_\alpha}(\phi_{v'}-\phi_{[v']_\alpha})\right]+\sigma_1\sigma_2 \E\left[\phi_{[v']_\alpha}(\phi_v-\phi_{[v]_\alpha})\right]\ .
\end{aligned}
\end{equation}
We first prove that the last two terms of \eqref{eqn: covariance psi} are positive.
By Lemma \ref{lem: GFF}, we can write $\phi_{[v]_\alpha}=\sum_{u\in\partial [v]_\alpha} p_{\alpha,v}(u)\  \phi_u$. 
Note that the vertices $u$ that are in $[v']_\alpha^c$ will not contribute
to the covariance $ \E\left[\phi_{[v]_\alpha}(\phi_{v'}-\phi_{[v']_\alpha})\right]$ by conditioning. Thus
$$
\begin{aligned}
 \E\left[\phi_{[v]_\alpha}(\phi_{v'}-\phi_{[v']_\alpha})\right]&=\sum_{u\in\partial [v]_\alpha\cap [v']_\alpha} p_{\alpha,v}(u)\  \E\left[\phi_u(\phi_{v'}-\phi_{[v']_\alpha})\right]\\
 &=\sum_{u\in\partial [v]_\alpha\cap [v']_\alpha} p_{\alpha,v}(u)\  \E\left[(\phi_u-\E[\phi_u|\F_{[v']_\alpha^c}])(\phi_{v'}-\E[\phi_{ v'}|\F_{[v']_\alpha^c}])\right]\ .
\end{aligned}
$$
Lemma  \ref{lem: green estimate} ensures that  the correlation in the sum are positive.

For the first term of \eqref{eqn: covariance psi}, the idea is to show that $\phi{[v]_\alpha}$ and $\phi_{\widetilde B}$ are close in the $L^2$-sense.
The same argument used to prove \eqref{eqn: fusion2} shows that
\begin{equation}
\label{eqn: lem12}
\E\left[\left(\phi_{[v]_\alpha} - \E[\phi_v| \F_{\widetilde B^c}]\right)^2\right]= O_N(1)\ .
\end{equation}
Moreover, since $v$ and $v_B$ are also at a distance smaller than $N^{1-\alpha}/100$ from each other, Lemma 12 in \cite{bolthausen-deuschel-giacomin} implies that 
\begin{equation}
\label{eqn: lem12b}
\E\left[\left(\phi_{\widetilde B}- \E[\phi_v| \F_{\widetilde B^c}]\right)^2\right]= O_N(1)\ .
\end{equation}  
 Equations \eqref{eqn: lem12} and \eqref{eqn: lem12b} give $\E[(\phi_{\widetilde B}-\phi_{[v]_\alpha})^2]= O_N(1)$ and similarly for $v'$. 
 All the above sum up to
 \begin{equation}
 \sigma_1^2\E\left[\phi_{[v]_\alpha}\phi_{[v']_\alpha}\right]=\sigma_1^2\E[\phi_{\widetilde B}^2]+O_N(1)\ .
 \end{equation}
 
 It remains to show that the second term of \eqref{eqn: covariance psi} is greater than $O_N(1)$. 
 Since $\phi_{[v]_\alpha}$ and $\phi_{[v']_\alpha}$ are $\F_{\widetilde B^c}$-measurable by definition of the box $\widetilde B$, we have the decomposition
$$
\begin{aligned}
\E\left[(\phi_v-\phi_{[v]_\alpha})(\phi_{v'}-\phi_{[v']_\alpha})\right]&=\E[(\phi_v-\E[\phi_v| \F_{\widetilde B^c}])(\phi_{v'}-\E[\phi_{v'}| \F_{\widetilde B^c}])] \\
&+ \E[(\E[\phi_v| \F_{\widetilde B^c}]-\phi_{[v]_\alpha})(\E[\phi_{v'}| \F_{\widetilde B^c}]-\phi_{[v']_\alpha})]\ .
\end{aligned}
$$
The first term is positive by Lemma \ref{lem: GFF}. As for the second, Equation \eqref{eqn: lem12} shows that
$$
\E\Big[\left(\E[\phi_v| \F_{\widetilde B^c}]-\phi_{[v]_\alpha}\right)\left(\E[\phi_{v'}|\F_{\widetilde B^c}]-\phi_{[v']_\alpha}\right)\Big]=O_N(1)\ .
$$
This concludes the proof of the lemma.
\end{proof}

Equation \eqref{eqn: corr bound} implies that the free energy of $\psi$ is smaller than the one of $\widetilde \psi$ by a standard comparison lemma, see Lemma \ref{lem: slepian} in the Appendix. It remains to prove an upper bound for the free energy of $\widetilde \psi$. 

Note that the field $\widetilde \psi$ is not a GREM {\it per se} because the variances of $g^{(1)}_B$, $B\in \mathcal B_\alpha$, are not the same for every $B$, as it depends on the distance of $B$ to the boundary. However, the variances of $\phi_{\widetilde B}$, $B\in \mathcal B_\alpha$, are uniformly bounded by $\frac{\alpha}{\pi}\log N^2+O_N(1)$; indeed
$$
\begin{aligned}
\E\left[\phi_{\widetilde B}^2\right]
&=\E\left[\phi_{v_B}^2\right] - \E\left[(\phi_{v_B}-\phi_{\widetilde B})^2\right]\\
&= \E\left[\phi_{v_B}^2\right] - \frac{1-\alpha}{\pi} \log N^2 + O_N(1)\\
&\leq \frac{1}{\pi}\log N^2  - \frac{1-\alpha}{\pi} \log N^2 + O_N(1)= \frac{\alpha}{\pi} \log N^2 + O_N(1),
\end{aligned}
$$
where we used Lemmas \ref{lem: GFF} and \ref{lem: green estimate} in the second line and Lemma \ref{lem: green estimate} in the third.

Moreover, note that for $v \in B$,
$$
\E[(g^{(2)}_v)^2]= \E[\psi_v^2] - \E[ (g^{(1)}_B)^2]= \sigma_1^2\big(\E[\phi_{[v]_\alpha}^2]-  \E[ \phi_{\widetilde B}^2]\big)+ \sigma_2^2\frac{1-\alpha}{\pi} \log N^2 - C\sigma_1^2\ .
$$
The first term is of order $O_N(1)$ by Equations \eqref{eqn: lem12} and \eqref{eqn: lem12b}. Thus one has
$$
\E[(g^{(2)}_v)^2]= \sigma_2^2\frac{1-\alpha}{\pi} \log N^2 +O_N(1)\ .
$$
The important point is that the variance of $g_v^{(2)}$ of $\widetilde \psi$ is uniform in $v$, up to lower order terms. 
Now consider the $2$-level GREM $(\bar \psi_v, v\in  A_{N,\rho})$
\begin{equation}
\bar \psi_v= \bar g^{(1)}_B + g^{(2)}_v 
\end{equation}
where $(g^{(2)}_v, v\in A_{N,\rho} )$ are as before and $(\bar g^{(1)}_B, B\in \mathcal B_\alpha)$ are i.i.d.~Gaussians of variance $\frac{\alpha}{\pi}\log N^2+O_N(1)$.
This field differs from $\widetilde \psi$ only from the fact that the variance of $ \bar g^{(1)}_B$ is the same for all $B$ and is the maximal variance of $(g^{(1)}_B, B\in \mathcal B_\alpha)$. 
The calculation of the free energy of $(\bar \psi_v, v\in  A_{N,\rho})$ is a standard computation and gives the correct upper bound in the statement of Theorem \ref{thm:freeenergyperturbed}.
(We refer to \cite{bolthausen-sznitman} for the detailed computation of the free energy of the GREM.) 
The fact that the free energy of $\bar \psi$ is larger than the one of $\widetilde \psi$ follows from the next lemma showing that the free energy of a hierarchical field is an increasing function of the variance of each point at the first  level. 

\begin{lemma}
Consider $N_1,N_2\in\N$. Let $(X^{(1)}_{v_1} , v_1\leq N_1)$ and $(X^{(2)}_{v_1,v_2} ; v_1\leq N_1, v_2\leq N_2)$.
Consider the Gaussian field of the form 
$$
X_v=\sigma_1(v_1) X^{(1)}_{v_1}+\sigma_2 X^{(1)}_{v_1,v_2}\ , \ \ v=(v_1,v_2)
$$ 
where $\sigma_2>0$ and $\sigma_1(v_1)>0$, $v_1\leq N_1$, might depend on $v_1$. 
Then $\E\left[\log \sum_v e^{\beta X_v}\right]$ is an increasing function in each variable $\sigma_1(v_1)$.
\end{lemma}
\begin{proof}
Direct differentiation gives
$$
\frac{\partial}{\partial \sigma_1(v_1)} \E\left[\log \sum_v e^{\beta X_v}\right]= \beta\E \left[\frac{\sum_{v_2} X_{v_1}e^{\beta X_{v_1,v_2}}}{Z_N(\beta)}\right]\ ,
$$
where $Z_N(\beta)=\sum_v e^{\beta X_v}$.
Gaussian integration by part then yields
$$
\beta\E \left[\frac{\sum_{v_2} e^{X_{v_1}\beta X_{v_1,v_2}}}{\sum_v e^{\beta X_v}}\right]
=\beta^2 \sigma_1(v_1)\E \left[ \frac{\sum_{v_2} e^{\beta X_{v_1,v_2}}}{Z_N(\beta)} - \frac{\sum_{v_2,v_2'} e^{\beta X_{v_1,v_2} }e^{\beta X_{v_1,v'_2}}}{Z_N(\beta)^2}\right]\ .
$$
The right side is clearly positive, hence proving the lemma.
\end{proof}

%%%%%%%%%%%%%%%%%%%%%%%%%%%%%%%%%%%%%%%%%%%%%%%%%%%%%%%%%%%%%%%%%%%%%%%%%%%%%%%

\subsection{Proof of the lower bound} 
Recall the definition of $V_{N}^\delta$ given in the introduction.
The two following propositions are used to compute the log-number of high points of the field $\psi$ in $V_{N}^\delta$.
The treatment follows the treatment of Daviaud \cite{daviaud} for the standard GFF.
The lower bound for the free energy is then computed using Laplace's method.
Define for simplicity $V_{12}:=\sigma_1^2\alpha+\sigma_2^2(1-\alpha)$.

\begin{proposition}
\label{prop:perturbed-maximum}
\begin{equation*}
\lim_{N \to \infty} \p\left(  \max_{v \in V_N^\delta} \psi_v  \ge  \sqrt{\frac{2}{\pi}} \gamma_{max}  \log N^2  \right) =0,
\end{equation*}
where 
\begin{equation*}
\gamma_{max}=\gamma_{max}(\alpha,\vec{\sigma}):=
\begin{cases}
\sqrt{V_{12}}, \ \  &\text{ if $\sigma_1\leq \sigma_2$,}
\\
\sigma_1\alpha+\sigma_2(1-\alpha),  \ \ &\text{ if $\sigma_1\geq \sigma_2$. }
\end{cases}
\end{equation*}
\end{proposition}
\begin{proof}
The case $\sigma_1\leq \sigma_2$ is direct by a union bound. 
In the case $\sigma_1\geq \sigma_2$, note that the field $\widetilde \psi$ defined in \eqref{eqn: tilde psi} but restricted to $V_N^\delta$ is a 2-level GREM with 
$cN^{2\alpha}$ (for some $c>0$) Gaussian variables of variance $\frac{\sigma_1^2\alpha}{\pi}\log N^2 + O_N(1)$ at the first level.
Indeed, for the field restricted to $V_N^\delta$, the variance of $\E[\phi_{\widetilde B}^2]$ is  $\frac{\sigma_1^2\alpha}{\pi}\log N^2 + O_N(1)$ by Lemma \ref{lem: green estimate}
since the distance to the boundary is a constant times $N$. Therefore, by Lemma \ref{lem: slepian} and Equation \eqref{eqn: corr bound}, we have
$$
\p\left(  \max_{v \in V_N^\delta} \psi_v  \ge  \sqrt{\frac{2}{\pi}} \gamma_{max}  \log N^2  \right) \leq \p\left(  \max_{v \in V_N^\delta} \widetilde \psi_v  \ge  \sqrt{\frac{2}{\pi}} \gamma_{max}  \log N^2  \right) \ .
$$
The result then follows from the maximal displacement of the 2-level GREM. We refer the reader to Theorem 1.1 in \cite{bovier-kurkova1} for the details.
\end{proof}

\begin{proposition}
\label{prop:perturbed-highpoints}
Let $\mathcal{H}_N^{\psi,\delta}(\gamma):=\left\{ v \in V_N^\delta: \,  \psi_v \ge  \sqrt{\frac{2}{\pi}} \gamma \log N^2 \right\}$ be the set of $\gamma$-high points within $V_N^\delta$ and define
$$
\begin{aligned}
&\text{if $\sigma\leq \sigma_2$} \qquad \mathcal{E}^{(\alpha,\vec{\sigma})}(\gamma):=1- \frac{\gamma^2}{V_{12}};\\
&\text{if $\sigma\geq \sigma_2$} \qquad \mathcal{E}^{(\alpha,\vec{\sigma})}(\gamma):=
\begin{cases}
1- \frac{\gamma^2}{V_{12}},
\ &\text{ if $ \gamma < \frac{V_{12}}{\sigma_1}$},
\\
(1-\alpha)  - \frac{(\gamma -\sigma_1\alpha)^2}{\sigma_2^2(1-\alpha)},
\ &\text{ if $ \gamma \geq  \frac{V_{12}}{\sigma_1}$.}
\end{cases}
\end{aligned}
$$
Then, for all $0<  \gamma < \gamma_{max},$ and for any $\mathcal{E}<\mathcal{E}^{(\alpha,\vec{\sigma})}(\gamma)$, there exists $c$ such that
\begin{equation}
\label{eqn: lower-}
\p\left(\vert  \mathcal{H}_N^{\psi,\delta}(\gamma)  \vert \le N^{2\mathcal{E}} \right) \le \exp \{- c (\log N)^2\}.
\end{equation}
\end{proposition}

Proposition \ref{prop:perturbed-highpoints} is obtained by a two-step recursion.
Two lemmas are needed. The first is a straightforward generalization of the lower bound in Daviaud's theorem (see Theorem 1.2 in \cite{daviaud} and its proof). 
For all $0<\alpha<1,$ denote by $\Pi_\alpha$ the centers of the square boxes in $\mathcal B_\alpha$ (as defined in Section \ref{sect: upperbound}) which also belong to $V_N^\delta$. 

 \begin{lemma} 
\label{lem:recurrence} 
Let $\alpha',\alpha'' \in (0,1]$ such that $0<\alpha'<\alpha'' \le \alpha$ or $\alpha \le \alpha'<\alpha'' \le 1.$ Denote by $\sigma$ the parameter $\sigma_1$ if $0<\alpha'<\alpha'' \le \alpha$ and by $\sigma$ the parameter $\sigma_2$ if $\alpha \le \alpha'<\alpha'' \le 1.$
Assume that the event
\begin{equation*}
\Xi:=\left\{\#\{v\in\Pi_{ \alpha'}:  \psi_v ( \alpha')  \geq  \gamma' \sqrt{\frac{2}{\pi}} \log N^2 \} \ge N^{\mathcal{E}'} \right\},
\end{equation*}
is such that
\begin{equation*}
\p(\Xi^c) \le \exp\{-c' (\log N)^2 \},
\end{equation*}
for some $\gamma' \ge 0$, $\mathcal{E}'>0$ and $c'>0$. 

Let \begin{equation*}
\mathcal{E}(\gamma):= \mathcal{E}' + (\alpha''-\alpha') - \frac{(\gamma-\gamma')^2}{ \sigma^2 (\alpha''-\alpha')}>0.
\end{equation*}
Then, for any $\gamma''$ such that $\mathcal{E}(\gamma'')>0$ and any $\mathcal{E}< \mathcal{E}(\gamma'')$, there exists $c$
such that
\begin{equation*}
\p\left(\#\{v\in\Pi_{ \alpha''}:  \psi_v ( \alpha'')    \geq \gamma'' \sqrt{\frac{2}{\pi}} \log N^2 \} \le N^{2\mathcal{E}}\right) \le \exp\{-c (\log N)^2 \} .
\end{equation*}
\end{lemma}
We stress that $\gamma''$ may be such that $\mathcal{E}(\gamma'')<\mathcal{E}'$.
The second lemma, which follows, serves as the starting point of the recursion and is proved in \cite{bolthausen-deuschel-giacomin} (see Lemma 8 in \cite{bolthausen-deuschel-giacomin}).
 \begin{lemma}
\label{lem:init}
For any $\alpha_0$ such that $0<\alpha_0<\alpha$, there exists $\mathcal E_0=\mathcal E_0(\alpha_0)>0$ and $c=c(\alpha_0)$ such that
$$
\p\left(\#\{v\in\Pi_{ \alpha_0}:  \psi_v ( \alpha_0) \geq 0\}\leq N^{\mathcal E_0}\right)\leq  \exp\{-c (\log N)^2 \} .
$$
\end{lemma}

\begin{proof}[Proof of Proposition \ref{prop:perturbed-highpoints}]
 Let $\gamma$ such that $0<  \gamma < \gamma_{max}$ and choose $\mathcal{E}$ such that $\mathcal{E}<\mathcal{E}^{(\alpha,\vec{\sigma})}(\gamma)$.  By Lemma \ref{lem:init}, for $\alpha_0<\alpha$ arbitrarily close to $0$, there exists $\mathcal E_0=\mathcal E_0(\alpha_0)>0$ and $c_0=c_0(\alpha_0)>0$, such that
\begin{equation}
\label{eqn: lower}
\p\left(\#\{v\in\Pi_{ \alpha_0}:  \psi_v ( \alpha_0) \geq 0\}\leq N^{2 \mathcal E_0}\right) \leq  \exp\{-c_0 (\log N)^2 \} .
\end{equation}
Moreover, let 
\begin{equation}
\label{eqn: E_1}
\mathcal{E}_1(\gamma_1):= \mathcal{E}_0 + (\alpha-\alpha_0) - \frac{\gamma_1^2}{ \sigma_1^2 (\alpha-\alpha_0)}  .
\end{equation}
Lemma \ref{lem:recurrence} is applied from $\alpha_0$ to $\alpha$. For any $\gamma_1$ with $\mathcal{E}_1(\gamma_1)>0$ and any 
$\mathcal{E}_1<\mathcal{E}_1(\gamma_1)$, there exists $c_1>0$ such that
$$
\p\left(\#\{v\in\Pi_{ \alpha}:  \psi_v ( \alpha)  \geq  \gamma_1 \sqrt{\frac{2}{\pi}} \log N^2 \}  \leq N^{2 \mathcal{E}_1} \right)\leq  \exp\{-c_1 (\log N)^2 \}.
$$
Therefore, Lemma \ref{lem:recurrence} can be applied again from $\alpha$ to $1$ for any $\gamma_1$ with $\mathcal{E}_1(\gamma_1)>0$. Define similarly
$
\mathcal{E}_2(\gamma_1,\gamma_2):= \mathcal{E}_1(\gamma_1) + (1-\alpha) - (\gamma_2-\gamma_1)^2/ \sigma_2^2 (1-\alpha).
$
Then, for any $\gamma_2$ with $\mathcal{E}_2(\gamma_1,\gamma_2)>0$, 
and $\mathcal{E}_2<\mathcal{E}_2(\gamma_1,\gamma_2)$, there exists $c_2>0$ such that
\begin{equation}
\label{eqn: lower prob}
\p\left(\#\{v\in V_N^\delta:  \psi_v   \geq  \gamma_2 \sqrt{\frac{2}{\pi}} \log N^2 \} \leq N^{2 \mathcal{E}_2} \right)\leq  \exp\{-c_2 (\log N)^2 \} .
\end{equation}
Observing that $0 \le \mathcal E_0 \le \alpha_0,$ Equation \eqref{eqn: lower-} follows from \eqref{eqn: lower prob} if it is proved that $\lim_{\alpha_0 \to 0}\mathcal{E}_2(\gamma_1,\gamma) = \mathcal{E}^{(\alpha,\vec{\sigma})}(\gamma)$ for an appropriate choice of $\gamma_1$ (in particular such that $\mathcal{E}_1(\gamma_1)>0$).
It is easily verified that, for a given $\gamma$, the quantity $\mathcal{E}_2(\gamma_1,\gamma)$ is maximized at
$
\gamma_1^*= \gamma \sigma_1^2(\alpha-\alpha_0)/(V_{12}-\sigma_1^2\alpha_0).
$
Plugging these back in \eqref{eqn: E_1} shows that $\mathcal{E}_1(\gamma_1^*)>0$ provided that
$
 \gamma< V_{12}/\sigma_1=:\gamma_{crit},
$
with $\alpha_0$ small enough (depending on $\gamma$).
Furthermore, since
$
\mathcal{E}_2(\gamma_1^*,\gamma)= \mathcal{E}_0 +(1-\alpha_0)- \gamma^2/(V_{12}-\sigma_1^2\alpha_0),
$
we obtain
$
\lim_{\alpha_0 \to 0} \mathcal{E}_2(\gamma_1^*,\gamma) = \mathcal{E}^{(\alpha,\vec{\sigma})}(\gamma),
$
which concludes the proof in the case $0<\gamma < \gamma_{crit}.$

If $\gamma_{crit} \leq \gamma < \gamma_{max}$, the condition $\mathcal{E}_1(\gamma_1^*)>0$ is violated as $\alpha_0$ goes to zero. However, the previous arguments can easily be adapted and we refer to subsection 3.1.2 in \cite{arguin-zindy} for more details.
\end{proof}

\begin{proof}[Proof of the lower bound of  Theorem \ref{thm:freeenergyperturbed}]
We will prove that for any $\nu>0$
$$
\p\left( f_{N,\rho}^{(\alpha,\vec{\sigma})}(\beta) \le f^{(\alpha,\vec{\sigma})}(\beta) - \nu \right) \longrightarrow 0, \qquad N \to 0.
$$
Define $\gamma_i:=i \gamma_{\max}/M$ for $0 \le i \le M$ ($M$ will be chosen large enough). Notice that Proposition \ref{prop:perturbed-maximum}, Proposition \ref{prop:perturbed-highpoints} and the symmetry property of centered Gaussian random variables imply that the event
\begin{eqnarray*}
\nonumber
 B_{N,M,\nu}&:=&\bigcap_{i=0}^{M-1} \left\{ \vert \mathcal{H}_N^{\psi,\delta}(\gamma_i) \vert   \ge  N^{2 \mathcal{E}^{(\alpha,\vec{\sigma})}(\gamma_{i})-\nu/3}  \right\}
 \bigcap \left\{  \max_{v \in V_N^\delta} \vert  \psi_v \vert  \le \sqrt{\frac{2}{\pi}} \gamma_{\max} \log N^2 \right\}
\end{eqnarray*}
satisfies
$$
\p( B_{N,M,\nu}) \longrightarrow 1, \qquad N \to \infty,
$$
for all $M \in \n^*$ and all $\nu>0.$ Then, observe that on $B_{N,M,\nu}$
\begin{eqnarray*}
Z_{N,\rho}^{(\alpha,\vec{\sigma})}(\beta) &\ge& \sum_{v \in V_N^\delta} \ee^{\beta \psi_v} \ge
 \sum_{i=1}^{M} ( \vert \mathcal{H}_N^{\psi,\delta}(\gamma_{i-1}) \vert- \vert \mathcal{H}_N^{\psi,\delta}(\gamma_i) \vert) N^{2 \sqrt{\frac{2}{\pi}} \gamma_{i-1} \beta}
\\
&=& \vert \mathcal{H}_N^{\psi,\delta}(0) \vert +   \Big(2 \sqrt{\frac{2}{\pi}} \frac{\gamma_{\max}}{M} \beta \log N\Big) \int_1^{M}  \vert  \mathcal{H}_N^{\psi,\delta}(\frac{\lfloor u \rfloor \gamma_{\max}}{M}) \vert  N^{2 \sqrt{\frac{2}{\pi}} \frac{u-1}{M}\gamma_{\max} \beta} d u
\\
&\ge&  \Big(2\sqrt{\frac{2}{\pi}} \frac{\gamma_{\max}}{M} \beta \log N\Big) \sum_{i=1}^{M-1}  \vert  \mathcal{H}_N^{\psi,\delta}( \gamma_i) \vert  N^{2 \sqrt{\frac{2}{\pi}} \gamma_{i-1}  \beta},
\end{eqnarray*}
where we used Abel's summation by parts formula.
Writing $\gamma_{i-1}=\gamma_{i}-\gamma_{\max}/M$ and $P_\beta(\gamma):= \mathcal{E}^{(\alpha,\vec{\sigma})}(\gamma) + \sqrt{\frac{2}{\pi}}  \beta \gamma,$ we get on $B_{N,M,\nu}$
\begin{equation}
\label{eqn: lower bound}
 f_{N,\rho}^{(\alpha,\vec{\sigma})}(\beta) \ge  
\frac{1}{\log N^2}\log  \left(\sum_{i=1}^{M-1}  N^{2 P_\beta(\gamma_{i})}\right) -\frac{\nu}{6}- \frac{ \sqrt{\frac{2}{\pi}} \gamma_{\max}\beta}{M}+o_N(1)\ .
\end{equation}

Using the expression of $ \mathcal{E}^{(\alpha,\vec{\sigma})}$ in Proposition \ref{prop:perturbed-highpoints} on the different intervals, 
it is easily checked by differentiation that
$
 \max_{\gamma \in \left[0,\gamma_{\max} \right]} P_{\beta}(\gamma)=f^{(\alpha,\vec{\sigma})}(\beta).
$
Furthermore, the continuity of $\gamma \mapsto P_{\beta}(\gamma)$ on $\left[0,\gamma_{\max} \right]$ yields 
\begin{equation*}
\max_{1 \le i \le M-1} P_{\beta}(\gamma_i) \longrightarrow \max_{\gamma \in \left[0,\gamma_{\max} \right]} P_{\beta}(\gamma)=f^{(\alpha,\vec{\sigma})}(\beta), \qquad M \to \infty.
\end{equation*}
Therefore, choosing $M$ large enough and applying Laplace's method in \eqref{eqn: lower bound} yield the result.
\end{proof}

%%%%%%%%%%%%%%%%%%%%%%%%%%%%%%%%%%%%%%%%%%%%%%%%%%%%%%%%%%%%%%%%%%%%%%%%%%%%%%%

\section{appendix}

The conditional expectation of the GFF has nice features such as the Markov property, see e.g. Theorems 1.2.1 and 1.2.2 in \cite{dynkin} for a general statement
on Markov fields constructed from symmetric Markov processes.
\begin{lemma}
\label{lem: GFF}
Let $B\subset A$ be subsets of $\Z^2$. Let $(\phi_v,v\in A)$ be a GFF on $A$. Then
$$
\E[\phi_v | \F_{B^c}]=\E[\phi_v | \F_{\partial B}], \qquad \forall v\in B,
$$
and 
$$
(\phi_v -\E[\phi_v | \F_{\partial B}], v\in B)
$$
has the law of a GFF on $B$.
Moreover, if $P_v$ is the law of a simple random walk starting at $v$ and $\tau_B$ is the first exit time of $B$, we have
$$
\E[\phi_v | \F_{\partial B}]=\sum_{u\in \partial B} P_v(S_{\tau_B}=u) ~\phi_u\ .
$$

\end{lemma}

The following estimate on the Green function can be found as Lemma 2.2 in \cite{ding} and is a combination of Proposition 4.6.2 and Theorem 4.4.4 in \cite{lawler-limic}.
\begin{lemma}
\label{lem: green estimate}
There exists a function $a:  \z^2 \times \z^2  \mapsto [0,\infty)$ of the form 
$$
a(v,v')= \frac{2}{\pi}\log \|v-v'\| +\frac{2\gamma_0 \log 8}{\pi} + O(\|v-v'\|^{-2})
$$
(where $\gamma_0$ denotes the Euler's constant) such that $a(v,v)=0$ and
$$
G_{A}(v,v')=E_{v}\left[a(v', S_{\tau_A})\right] - a(v,v')\ .
$$
\end{lemma}

Slepian's comparison lemma can be found in \cite{ledoux-talagrand} and in \cite{kahane} for the result on log-partition function.
\begin{lemma}
\label{lem: slepian}
Let $(X_1,\cdots, X_N)$ and $(Y_1,\cdots, Y_N)$ be two centered Gaussian vectors in $N$ variables such that
$$
\E[X_i^2]= \E[Y_i^2]\  \forall i,  \qquad \E[X_iX_j]\geq \E[Y_iY_j] \ \forall i\neq j\ .
$$
Then for all $\beta>0$
$$
\E\left[\log \sum_{i=1}^N e^{\beta X_i}\right]\leq \E\left[\log \sum_{i=1}^N e^{\beta Y_i}\right]\ ,
$$
and for all $\lambda>0$,
$$
\p\left(\max_{i=1,\dots,N} X_i > \lambda\right) \leq \p\left(\max_{i=1,\dots,N}  Y_i > \lambda\right) \ .
$$
\end{lemma}

\bigskip

{\noindent\bf Acknowledgements:} The authors would like to thank the Centre International de Rencontres Math\'ematiques in Luminy for hospitality and financial support during part of this work.

\end{document}